\documentclass[a4paper,12pt]{article}

\usepackage[a4paper]{geometry}
\usepackage{color}
\usepackage{graphicx}
\usepackage{amsmath}
\usepackage{amsfonts}
\usepackage{amsthm}
\usepackage{amscd}
\usepackage{epsfig}
\usepackage{amssymb}
\usepackage{slashed}
\usepackage{tabularx}
\usepackage{calligra}
\usepackage{enumerate}
\usepackage{hyperref}
\usepackage{float}
\usepackage{stackrel}
\usepackage[T1]{fontenc}
\usepackage{longtable}
\usepackage{amsthm}
\usepackage{mathrsfs}
\usepackage{verbatim} 
\usepackage{fancyhdr}
\usepackage{tikz}
\usepackage{tikz-cd}
\usetikzlibrary{matrix}
\usetikzlibrary{positioning}
\usepackage[all]{xy}
\usepackage{units}
\usepackage{textcomp}
\usepackage{turnstile}
\usepackage{vmargin}
\usepackage{anysize}
\usepackage{ stmaryrd }
\usepackage{tikz-3dplot}

\setmargins{2,5cm}       
{1,5cm}                        
{16cm}                      
{23,42cm}                    
{10pt}                           
{1cm}                           
{0pt}                             
{2cm}                           

\theoremstyle{plain}

\newtheorem{proposition}{Proposition}

\numberwithin{equation}{section}

\def\d{{\rm d}}
\def\i{{\rm i}}
\def\CP{\mathbb{CP}}
\def\CM{\mathbb{CM}}
\def\PT{\mathbb{PT}}

\begin{document}

\title{Hyper-K\"ahler instantons, symmetries, and flat spaces}
\author{Bernardo Araneda\footnote{\texttt{bernardo.araneda@aei.mpg.de}}  \\
Max-Planck-Institut f\"ur Gravitationsphysik \\ 
(Albert-Einstein-Institut), Am M\"uhlenberg 1, \\
D-14476 Potsdam, Germany}

\date{\today}

\maketitle

\begin{abstract}
We find all hyper-K\"ahler 4-manifolds admitting conformal K\"ahler structures with respect to either orientation, and we show that these structures can be expressed as a combination of twistor elementary states (and possibly a self-dual dyon) in locally flat spaces. The complex structures of different flat pieces are  not compatible however, reflecting that the global geometry is not a linear superposition. For either orientation the space must be Gibbons-Hawking (thus excluding the Atiyah-Hitchin metric), and, if the orientations are opposite, it must also be toric and have an irreducible Killing tensor. We also show that the only hyper-K\"ahler 4-metric with a non-constant Killing-Yano tensor is the half-flat Taub-NUT instanton.
\end{abstract}

\section{Introduction}

The path integral for Euclidean Quantum Gravity is expected to be dominated by gravitational instantons \cite{Hawking76, GibbonsHawking2} (Ricci-flat, complete Riemannian 4-manifolds). Recent developments \cite{AA, BG, AADNS, BGL, Li} suggest that the existence of compatible, integrable complex structures in these spaces is one of their central features. Due to Ricci-flatness, this is equivalent to conformally K\"ahler, reducing to hyper-K\"ahler if the conformal factor is constant. A {\em conformal} K\"ahler structure introduces some features which are not present  in the purely K\"ahler case (e.g. isometries), or which require a different treatment (e.g. twistor descriptions). In this note we are interested in instantons that are both hyper-K\"ahler and conformally K\"ahler, where the two structures may have the same or opposite orientation. Our focus is on symmetries, twistor aspects and relations to flat space structures.

Tod and Ward showed in \cite{TodWard} that if an anti-self-dual vacuum space, i.e. one in which 
\begin{align}
R_{ab}=0, \qquad \tilde{\Psi}_{A'B'C'D'}=0 \label{HK0}
\end{align}
(where $R_{ab}$ is the Ricci tensor and $\tilde{\Psi}_{A'B'C'D'}$ the self-dual Weyl curvature spinor) admits a non-constant, valence-2 primed Killing spinor, then it is necessarily given, on positive-definite sections, by the Gibbons-Hawking ansatz. In view of subsequent developments \cite{Pontecorvo, DT}, we can restate this result as saying that a hyper-K\"ahler 4-manifold with an equally oriented conformal K\"ahler structure must be Gibbons-Hawking.

Requiring completeness of the Gibbons-Hawking space leads to a gravitational multi-instanton \cite{GibbonsHawking}. These metrics are obtained by appropriately combining flat 4-space metrics. Page showed \cite{Page79} that also the Green's functions for various fields can be obtained by a suitable combination of flat 4-space pieces. In this note we wish to show that the conformal K\"ahler structure is similarly given by suitably combining flat pieces, each corresponding to a twistor elementary state\footnote{A conformal K\"ahler structure has an appealing physical interpretation when formulated in twistor terms, which is why we are interested in these descriptions; see section \ref{Sec:KSTT}.} in the asymptotically locally Euclidean (ALE) case, and the same plus a self-dual dyon in the asymptotically locally flat (ALF) case. However, we will show that the complex structures of different flat pieces are not compatible, which we interpret as manifesting the fact that the global geometry is not a linear superposition of elementary states but the instantons ``interact''.

On the other hand, the other possibility is that the conformal K\"ahler orientation is opposite to the hyper-K\"ahler one. In this case, we will show that the metric must again be Gibbons-Hawking (and so, conformally K\"ahler on both sides), but with more symmetries: there are two commuting Killing vectors and an irreducible Killing tensor, although generally no non-constant Killing-Yano tensor. (In fact, we will see that the only hyper-K\"ahler metric with a non-constant Killing-Yano tensor is half-flat Taub-NUT.) Furthermore, by comparison to the literature \cite{PH, Woodhouse}, the geometry must be an ALE two-centred multi-instanton. In this case we find a single flat space compatible with the complex structure, and we show that the twistor description is indeed not a superposition of elementary states but rather the Pleba\'nski-Demia\'nski field \cite{PD}, consistently with \cite{PH}.

Since we see from the above that a hyper-K\"ahler 4-metric admitting a conformal K\"ahler structure w.r.t either orientation must be Gibbons-Hawking, it follows that 
the Atiyah-Hitchin metric \cite{AtiyahHitchin} is not conformally K\"ahler; in particular, it does not admit (non-constant) Killing spinors or conformal Killing-Yano tensors\footnote{See  \cite{GibbonsRuback1, GibbonsRuback2} for a detailed description of the hidden symmetries of Taub-NUT, where possible generalisations to Atiyah-Hitchin are one of the main motivations.}. 

Finally, part of our motivation comes from the Chen-Teo solution \cite{ChenTeo2}, which is a generalisation of Pleba\'nski-Demia\'nski, it is conformally K\"ahler \cite{AA}, and it has an ALE three-centred Gibbons-Hawking space as one of its hyper-K\"ahler limits; however, we were not able to find a global flat space compatible with the complex structure of three centres.

We will use the abstract index notation following Penrose and Rindler \cite{PR2}. Our main results are summarised by propositions throughout the text.

\section{Killing spinors and twistor theory}
\label{Sec:KSTT}

Let $(M,g_{ab})$ be a 4-dimensional (4d), orientable Riemannian manifold. We say that $(M,g_{ab})$ is K\"ahler if there is an integrable complex structure $J$ compatible with $g$ (i.e. $g(J\cdot,J\cdot)=g$) and $\d\kappa=0$, where $\kappa_{ab}:=g_{bc}J^{c}{}_{a}$ is the fundamental 2-form. The latter necessarily has definite duality under the Hodge star $*_{g}$ of $g_{ab}$ (regardless of the K\"ahler condition). A 2-form $\omega$ is self-dual (SD) if $*_{g}\omega=\omega$ and anti-self-dual (ASD) if $*_{g}\omega=-\omega$. For $\kappa$, we refer to this distinction as the orientation of $\kappa$ (or of $J$). The geometry is conformally K\"ahler if there is a non-zero scalar field $\Omega$ such that $\hat{g}_{ab}=\Omega^{2}g_{ab}$ is K\"ahler. The 2-form $\hat\kappa_{ab}=\Omega^2\kappa_{ab}$ is called symplectic form.

As shown in Lemma 2.1 of \cite{DT}, $(M,g_{ab})$ is conformally K\"ahler if and only if it admits a valence-2 real Killing spinor. Choosing SD orientation for concreteness, this means that there is a primed spinor field $K_{A'B'}=K_{(A'B')}$ with\footnote{We denote Euclidean spinor conjugation by $\dagger$.} $K_{A'B'}=K_{A'B'}^{\dagger}$ and 
\begin{align}
\nabla_{AA'}K_{B'C'} = 2\epsilon_{A'(B'}t_{C')A} \label{primedKS}
\end{align}
where $t_{AA'}=-\frac{1}{3}\nabla_{A}^{B'}K_{A'B'}\neq0$. The tensorial version of \eqref{primedKS} is the conformal Killing-Yano (CKY) equation, $\nabla_{a}Z_{bc} = \nabla_{[a}Z_{bc]} - 2g_{a[b}t_{c]}$, with $Z_{ab}=K_{A'B'}\epsilon_{AB}$. Given a solution to \eqref{primedKS}, the K\"ahler metric is $\hat{g}_{ab}=\Omega^{2}_{+}g_{ab}$, where the conformal factor $\Omega_{+}$ and the complex structure are
\begin{align}
 \Omega^{-2}_{+} = \frac{K_{A'B'}K^{A'B'}}{2}, \qquad 
 (J_{+})^{a}{}_{b} = \Omega_{+} \, K^{A'}_{B'}\delta^{A}_{B}.
 \label{CFSD}
\end{align}
The SD symplectic form is $\hat\kappa^{+}_{ab}=\hat{g}_{bc}(J_{+})^{c}{}_{a}=\Omega_{+}^{3}K_{A'B'}\epsilon_{AB}$. The vector field $t^{a}$ preserves (Lie-drags) both $\Omega_{+}$ and $\hat{\kappa}^{+}$, so we refer to it as the `Hamiltonian vector field'. However, $t^{a}$ is holomorphic (i.e. it preserves $J_{+}$) if and only if it is Killing, which can be shown to hold if and only if the Ricci tensor of $g_{ab}$ is invariant under $J_{+}$. 

For ASD orientation, all of the above has an analogue in terms of unprimed spinors:
\begin{align}
\nabla_{AA'}\sigma_{BC}=2\epsilon_{A(B}\xi_{C)A'}, \qquad 
 \Omega^{-2}_{-} = \frac{\sigma_{AB}\sigma^{AB}}{2}, \qquad 
 (J_{-})^{a}{}_{b} = \Omega_{-} \, \sigma^{A}_{B}\delta^{A'}_{B'}. 
 \label{unprimedKS}
\end{align}
The ASD symplectic form is $\hat\kappa^{-}_{ab}=\Omega_{-}^{3}\sigma_{AB}\epsilon_{A'B'}$.

A special case of CKY tensors are Killing-Yano (KY) tensors: 2-forms $Y_{ab}=Y_{[ab]}$ satisfying $\nabla_{(a}Y_{b)c}=0$, or equivalently $\nabla_{a}Y_{bc}=\nabla_{[a}Y_{bc]}$. We have the following:

\begin{proposition}\label{prop:KY}
A 4d geometry admits a non-constant KY tensor if and only if it is conformally K\"ahler with respect to both orientations and the corresponding Hamiltonian vector fields coincide. 
\end{proposition}

\begin{proof}
The KY equation $\nabla_{a}Y_{bc}=\nabla_{[a}Y_{bc]}$ can be written as $\nabla_{a}Y_{bc}=-2\varepsilon_{abcd}\xi^{d}$, where $\xi_{a}=\frac{1}{6}\nabla^{b} *_{g}Y_{ab}$ and $\varepsilon_{abcd}$ is the volume form. In spinors, we can decompose into SD and ASD parts as $Y_{ab}=K_{A'B'}\epsilon_{AB}-\sigma_{AB}\epsilon_{A'B'}$ and $*_{g}Y_{ab}=K_{A'B'}\epsilon_{AB}+\sigma_{AB}\epsilon_{A'B'}$, and the KY equation $\nabla_{a}Y_{bc}=-2\varepsilon_{abcd}\xi^{d}$ is equivalent to
\begin{align}
 \nabla_{AA'}\sigma_{BC} = 2\epsilon_{A(B}\xi_{C)A'}, \qquad 
 \nabla_{AA'}K_{B'C'} = 2\epsilon_{A'(B'}\xi_{C')A},
 \label{KYEq}
\end{align}
where we used that $\varepsilon_{abcd}=\epsilon_{AC}\epsilon_{BD}\epsilon_{A'D'}\epsilon_{B'C'} - \epsilon_{AD}\epsilon_{BC}\epsilon_{A'C'}\epsilon_{B'D'}$.
\end{proof}

A symmetric spinor $K_{A'B'}$ can always be decomposed into principal spinors as 
\begin{align}
K_{A'B'}=\alpha_{(A'}\beta_{B')}, \label{principalspinors}
\end{align}
where $\alpha_{A'}$ and $\beta_{A'}$ are linearly independent as long as $K^{A'B'}K_{A'B'}\neq0$. In fact $\beta_{A'}=\i\alpha_{A'}^{\dagger}$ if $K_{A'B'}$ is real. Equations \eqref{principalspinors} and \eqref{primedKS} imply that $\alpha_{A'}$ is {\em shear-free}: 
\begin{align}
\alpha^{A'}\alpha^{B'}\nabla_{AA'}\alpha_{B'}=0 \label{SFR0}
\end{align}
(analogously for $\beta_{A'}$). This equation is invariant under $\alpha_{A'}\to \lambda \alpha_{A'}$, and it is equivalent to integrability of the complex structure $J_{+}$. Eq. \eqref{SFR0} can be rewritten as
\begin{align}
\nabla_{A(A'}\alpha_{B')} = \mathcal{A}_{A(A'}\alpha_{B')}, 
\label{SFR}
\end{align}
for some 1-form $\mathcal{A}_{a}$. Under $\alpha_{A'}\to \lambda \alpha_{A'}$, we have $\mathcal{A}_{a} \to \mathcal{A}_{a} + \lambda^{-1}\nabla_{a}\lambda$. Thus, $\alpha_{A'}$ satisfies a charged twistor equation w.r.t. the covariant derivative $\nabla_{a} - \mathcal{A}_{a}$. This is a connection on a principal bundle over $M$ encoding the gauge symmetry $\alpha_{A'}\to \lambda \alpha_{A'}$, $\beta_{A'}\to \lambda^{-1} \beta_{A'}$. The structure group is $G=\mathbb{C}^{\times}$ in the complexified case, and $G=U(1)$ on Riemannian sections. The connection is not necessarily (A)SD, but it will be so for the geometries considered in this note, see proposition \ref{prop:ASDC} below.

The geometry $(M,g_{ab})$ is called hyper-K\"ahler iff it is K\"ahler for a 2-sphere of complex structures. In 4d, this is equivalent to (anti-)self-duality of the Riemann tensor, $*_{g}R_{abcd}=\pm R_{abcd}$. For concreteness let us choose $*_{g}R_{abcd}=-R_{abcd}$, then this is the same as eq. \eqref{HK0}. Equivalently, there is a parallel primed spinor field, 
\begin{align}
\nabla_{AA'}o_{B'}=0.  \label{parallelspinor}
\end{align}
The complex conjugate $o^{\dagger}_{A'}$ is also parallel. We choose the normalisation $o_{A'}o^{\dagger A'}=1$. Equations \eqref{HK0} or \eqref{parallelspinor} are also equivalent to the existence of a basis of parallel SD 2-forms satisfying the quaternion algebra. We refer to this as a (parallel) hyper-Hermitian structure. We can choose the SD 2-forms to be
\begin{align}
\omega^{+}_{1\, ab} = 2\i o_{(A'}o^{\dagger}_{B')}\epsilon_{AB}, \quad
\omega^{+}_{2\, ab} = \i(o^{\dagger}_{A'}o^{\dagger}_{B'}-o_{A'}o_{B'})\epsilon_{AB}, \quad
\omega^{+}_{3\, ab} = (o_{A'}o_{B'}+o^{\dagger}_{A'}o^{\dagger}_{B'})\epsilon_{AB}.
\label{ASDbasis}
\end{align}

\begin{proposition}\label{prop:ASDC}
Let $(M,g_{ab})$ be hyper-K\"ahler \eqref{HK0}, and assume it is also conformally K\"ahler with the same orientation. Then the $U(1)$ connection defined by \eqref{SFR} is anti-self-dual, $\nabla_{A(A'}\mathcal{A}^{A}_{B')}=0$ \footnote{As the proof shows, this is actually true as long as $\tilde\Psi_{A'B'C'D'}\alpha^{D'}=0$. For Riemann signature this implies only conformally half-flat $\tilde\Psi_{A'B'C'D'}\equiv0$ (rather than \eqref{HK0}), while for other signatures $\tilde\Psi_{A'B'C'D'}$ can be Petrov type N.}.
\end{proposition}

\begin{proof}
Apply a covariant derivative to \eqref{SFR}, contract over unprimed indices, and symmetrise over primed indices: $\nabla_{A(A'}\nabla^{A}_{B'}\alpha_{C')} = 
(\nabla_{A(A'}\mathcal{A}^{A}_{B'})\alpha_{C')}$. The left hand side is $\Box_{(A'B'}\alpha_{C')}=-\tilde{\Psi}_{A'B'C'}{}^{D'}\alpha_{D'}$, so it vanishes for $\tilde\Psi_{A'B'C'D'}\equiv0$.
\end{proof}

Putting everything together, we see that a conformal K\"ahler structure on a hyper-K\"ahler manifold with the same orientation is equivalent to a charged twistor with respect to an ASD Maxwell field. Other applications of charged twistors include the Teukolsky equations \cite{A18}, and scattering amplitudes on Taub-NUT \cite{Adamo}. Hyper-K\"ahler manifolds with higher-valence Killing spinors have been thoroughly studied in \cite{DunajskiMason}.

\medskip
There is another way of describing the relation between K\"ahler structures and twistor theory \cite{Pontecorvo}, that is also relevant for this work as it provides some physical intuition. To review this construction, we first recall that a geometry $(M,g_{ab})$ satisfying \eqref{HK0} has a twistor space, which we denote by $\mathcal{PT}$. (We refer to \cite{Penrose76, HuggettTod, WardWells} for background on twistor theory relevant to this discussion.) This is a 3d complex manifold, each point of which is a totally null, SD 2-surface (called {\em $\alpha$-surface}) in the complexification $\mathbb{C}M$. In turn, a space-time point $x$ is identified with a rational curve $L_x\cong\CP^1$ in $\mathcal{PT}$, called twistor line. This correspondence is called incidence relation. There is a double fibration $\mathcal{PT}\leftarrow\mathcal{F}\to\mathbb{C}M$, where $\mathcal{F}$ is the projective primed spin bundle, with local coordinates $(x^{a},\pi_{A'})$. The charged twistors of eq. \eqref{SFR} correspond to 2d systems of $\alpha$-surfaces in $\mathbb{C}M$ \footnote{This is because the integrability condition for $\alpha$-surfaces is precisely the shear-free eq. \eqref{SFR}.}. The corresponding subset of $\mathcal{PT}$ is, at least when restricted to rational curves, a  quadric surface: eq. \eqref{SFR} is equivalent to \eqref{primedKS} (assuming Ricci-flatness and Riemannian signature), which in turn is equivalent to asking the function
\begin{align}
 \chi|_{L_x}:=F(x^{a},\pi_{A'}) := K^{A'B'}\pi_{A'}\pi_{B'} \label{quadric0}
\end{align}
to satisfy $\pi^{A'}\nabla_{AA'}F=0$. So we see that the zero set of $\chi$ is a quadric (on rational curves). Alternatively, there is an equivalence \cite{Pontecorvo} between (scalar-flat) K\"ahler metrics and holomorphic sections of the square-root of the anti-canonical bundle of $\mathcal{PT}$ with two zeroes on every twistor line. 

The above correspondence provides some physical insights when understood in complexified flat space-time \cite{A22}. We denote the latter by $\CM$, with a flat holomorphic metric $\eta=\d{t}^2-\d{x}^2-\d{y}^2-\d{z}^2$. The twistor space of $\CM$ is $\PT=\CP^3\backslash\CP^1$, with homogeneous coordinates $Z^{\alpha}=(\omega^{A},\pi_{A'})$. The incidence relation is 
\begin{align}
\omega^{A}=\i x^{AA'}\pi_{A'}, \label{IR}
\end{align}
where $x^{00'}=\frac{(t+z)}{\sqrt{2}}$, $x^{11'}=\frac{(t-z)}{\sqrt{2}}$, $x^{01'}=\frac{(x+\i y)}{\sqrt{2}}$, $x^{10'}=\frac{(x-\i y)}{\sqrt{2}}$.
The twistor line $L_x$ of $x\in\CM$ is a holomorphic $\CP^1\subset\PT$ obtained by fixing $x^{AA'}$ in \eqref{IR}. The $\CP^1$ removed in the definition of $\PT$ can be interpreted as the twistor line of the point corresponding to conformal infinity. A twistor quadric is a surface $\mathbb{Q}=\{Z^{\alpha}\in\PT \,|\, \chi(Z^{\alpha})=0\}$, where $\chi(Z^{\alpha}):=Q_{\alpha\beta}Z^{\alpha}Z^{\beta}$ for some matrix $Q_{\alpha\beta}$, which we assume not to be of the form $A_{\alpha}A_{\beta}$ for some $A_{\alpha}$. To induce a conformal K\"ahler structure in $\CM$ from $\mathbb{Q}$, we use the Penrose transform: for any point $x$ in an open subset $U\subset\CM$, define
\begin{align}
\phi_{A'_1...A'_{2h}}(x) = \frac{1}{2\pi\i}\oint_{\Gamma}\chi^{-(h+1)}\big|_{L_{x}} 
\pi_{A'_1}...\pi_{A'_{2h}}\pi_{B'}\d\pi^{B'},
\label{PenroseTransform}
\end{align}
where $h$ is a non-negative integer, $L_x\cong\CP^1$ is the twistor line of $x$ and $\Gamma$ is a contour in $L_x$ separating the singularities of $\chi^{-1}$. For $h=0$ we set $\Omega_{+}:=\phi$. Letting $\hat{g}_{ab}:=\Omega_{+}^{2}\eta_{ab}$ and $\hat\kappa^{+}_{ab}:=\phi_{A'B'}\epsilon_{AB}$, the pair $(\hat{g}_{ab},\hat\kappa^{+}_{ab})$ is a (complexified) K\"ahler structure in $U$. Physically, $\mathbb{Q}$ defines zero-rest-mass free fields \eqref{PenroseTransform} on space-time for any $h$. 

Conversely: 
\begin{proposition}\label{prop:CF}
Given a conformally flat K\"ahler structure $(\hat{g}_{ab},\hat\kappa^{+}_{ab})$, there exists a twistor quadric $\mathbb{Q}$ from which $(\hat{g}_{ab},\hat\kappa^{+}_{ab})$ locally arises via the above construction. This quadric can be found by focusing on the conformal factor $\Omega_{+}$: if a $\chi$ is found such that $\Omega_{+}$ is given by \eqref{PenroseTransform} with $h=0$, then $\hat\kappa^{+}_{ab}$ must necessarily be given by \eqref{PenroseTransform} with $h=1$. 
Furthermore, the general conformal factor is given by
\begin{equation}
\begin{aligned}
 2\, \Omega_{+}^{-2} ={}& \tfrac{1}{4}a_{AB}a^{AB}(x_ax^a)^2 + 2\i a_{AB}b^{B}_{A'}x^{AA'}x_cx^c - 2a_{AB}c_{A'B'}x^{AA'}x^{BB'} \\ 
& - 2[(b_ab^a)(x_bx^b)-(b_ax^a)^2] + 4\i b_{A}^{B'}c_{A'B'}x^{AA'} + c^{A'B'}c_{A'B'},
\label{FormulaCF}
\end{aligned}
\end{equation}
for some constant spinors $a_{AB},b^{A'}_{B},c^{A'B'}$ (the components of $Q_{\alpha\beta}$).
\end{proposition}

\begin{proof}
Recall that $\hat\kappa^{+}_{ab}$ can be written as $\hat\kappa^{+}_{ab}=\Omega_{+}^{3}K_{A'B'}\epsilon_{AB}$, where $K_{A'B'}$ is a Killing spinor and $\Omega_{+}$ is given by eq. \eqref{CFSD}. Since $\hat\kappa^{+}_{ab}$ is SD and closed, $\Omega_{+}^{3}K_{A'B'}$ is a spin 1 field, so it can be generated by the Penrose transform for some twistor function. At this point we do not know the form of this twistor function. Now, let $\chi^{-2}|_{L_x}:=4\i(K^{A'B'}\pi_{A'}\pi_{B'})^{-2}$, and replace this in the right hand side of formula \eqref{PenroseTransform} with $h=1$. To compute the integral, decompose $K^{A'B'}$ into principal spinors as in \eqref{principalspinors}, and use proposition 1 in \cite{A23a} (with $r=s=2$). Noting that $\Omega_{+}=(\frac{1}{2}K_{A'B'}K^{A'B'})^{-1/2}=-2\i(\alpha_{A'}\beta^{A'})^{-1}$, the result of the integral is exactly $\Omega_{+}^{3}K_{A'B'}$, which is the field we wanted to recover. Since $\chi|_{L_x}$, as just defined, is the restriction to $L_x$ of the quadratic function $\chi=Q_{\alpha\beta}Z^{\alpha}Z^{\beta}$ for some $Q_{\alpha\beta}$, we see that the 2-form $\hat\kappa^{+}_{ab}$ arises from a twistor quadric $\mathbb{Q}$. The complex structure $(J_{+})^{a}{}_{b}$ is defined by pairs of points in $\mathbb{Q}$, and the metric is $\hat{g}_{ab}=\hat\kappa_{bc}(J_{+})^{c}{}_{a}$.

Replacing $2\i(K^{A'B'}\pi_{A'}\pi_{B'})^{-1}$ in the right hand side of \eqref{PenroseTransform} with $h=0$ and computing the integral, one gets $-2\i(\alpha_{A'}\beta^{A'})^{-1}$, which is exactly $\Omega_{+}$. This proves that if one knows $\Omega_{+}$ and $\hat\kappa^{+}_{ab}$, and finds the $\chi^{-1}$ that generates $\Omega_{+}$, then $\chi^{-2}$ generates $\hat\kappa^{+}_{ab}$. Finally, formula \eqref{FormulaCF} follows after noticing that a generic expression for $K^{A'B'}$ is obtained by evaluating $\chi = Q_{\alpha\beta}Z^{\alpha}Z^{\beta}$ on a twistor line $L_x$ and using the incidence relation \eqref{IR}.
\end{proof}

Finally, there are conformal K\"ahler structures for which the ASD Maxwell field associated to charged twistors \eqref{SFR} via prop. \ref{prop:ASDC} vanishes: these are the {\em elementary states},
\begin{align}
 \chi = (A_{\alpha}Z^{\alpha})(B_{\beta}Z^{\beta}) \label{ES}
\end{align}
for some $A_{\alpha}, B_{\alpha}$. The quadric \eqref{ES} is the union of two planes in $\PT$. The corresponding space-time fields \eqref{PenroseTransform} are singular on the light-cone of a point $q\in\CM$, whose twistor line $L_q$ is the intersection line of the planes; see \cite[Section 6.10]{PR2}. One can impose reality conditions in such a way that this light-cone does not intersect a real Lorentzian section of $\CM$, so the fields \eqref{PenroseTransform} will be regular there (these are sometimes called hopfions or knotted fields). 
Alternatively, $L_{q}$ can be interpreted as corresponding to conformal infinity, and then chosen to be the $\CP^1$ removed in the definition of $\PT$.

\section{Gibbons-Hawking multi-instantons}
\label{Sec:HKO}

Let $(M,g_{ab})$ be hyper-K\"ahler, and assume that it admits a conformal K\"ahler structure with the same orientation. This is the same as asking $(M,g_{ab})$ to admit both a parallel spinor $o_{A'}$ \eqref{parallelspinor} and a primed Killing spinor $K_{A'B'}$ \eqref{primedKS} with $t_{a}\neq0$. From \eqref{primedKS} we deduce that\footnote{We denote contractions with $o^{A'}$ and $o^{\dagger A'}$ by the subscripts $0'$ and $1'$ respectively.} $\nabla_{a}K_{0'0'} = 2o_{A'}o_{B'}t^{B'}_{A}$, $\nabla_{a}K_{0'1'} = 2o_{(A'}o^{\dagger}_{B')}t^{B'}_{A}$, $\nabla_{a}K_{1'1'} = 2o^{\dagger}_{A'}o^{\dagger}_{B'}t^{B'}_{A}$. Using \eqref{ASDbasis}, this is equivalent to 
$\d(\i K_{0'1'}) = t\lrcorner\,\omega^{+}_1$, 
$\d[\frac{\i}{2}(K_{1'1'}-K_{0'0'})]=t\lrcorner\,\omega^{+}_2$,
$\d[\frac{1}{2}(K_{0'0'}+K_{1'1'})]=t\lrcorner\,\omega^{+}_3$. 
Recalling Cartan's formula for the Lie derivative, and using that $\d\omega^{+}_i=0$ and that $t\lrcorner\,\omega^{+}_i$ is exact for all $i=1,2,3$, we get $\pounds_{t}\omega^{+}_i=0$. Since the Ricci-flat condition implies $\nabla_{(a}t_{b)}=0$, we see that $t^{a}$ is a tri-holomorphic Killing vector, thus, the metric must be given by the Gibbons-Hawking ansatz \cite{GibbonsRuback2}. To check this, define
\begin{align}
 V^{-1}:=t_at^a, \quad 
 x_1:=\i K_{0'1'}, \quad 
 x_2:=\tfrac{1}{2}\i(K_{1'1'}-K_{0'0'}), \quad
 x_3:=\tfrac{1}{2}(K_{0'0'}+K_{1'1'}), 
 \label{GHcoordinates}
\end{align}
a coordinate $\tau$ by $t^{a}\partial_{a}=\partial_{\tau}$, and a 1-form $A$ by $t_{a}\d{x}^{a}=V^{-1}(\d\tau+A)$. Then, using (for example) the generic expression for the metric in terms of a hyper-Hermitian structure given in \cite[Appendix A]{A23b}, we get
\begin{align}
 g = V^{-1}(\d\tau+A)^2+V(\d{x}_1^2+\d{x}_2^2+\d{x}_3^2), \label{GH}
\end{align}
which is the Gibbons-Hawking ansatz.
The hyper-K\"ahler condition is $\d{A} = *_{3}\d{V}$, where $*_{3}$ is the Hogde star in flat Euclidean 3-space. The function $V$ satisfies the flat 3d Laplace equation $\Delta_{3}V=0$, and any solution to this equation can be used to construct a hyper-K\"ahler metric via \eqref{GH}. Using the definitions \eqref{GHcoordinates}, the Killing spinor, conformal factor, and conformal K\"ahler form are computed to be, respectively:
\begin{subequations}
\begin{align}
 K_{A'B'} ={}& (x_3-\i x_2)o_{A'}o_{B'} + 2\i x_1 o_{(A'}o^{\dagger}_{B')} + (x_3+\i x_2) o^{\dagger}_{A'}o^{\dagger}_{B'},  \label{KSexplicit} \\
 \Omega_{+} ={}& (x_1^2 + x_2^2 + x_3^2)^{-1/2}=: r^{-1}, \\
 \hat\kappa^{+} ={}& \frac{1}{r^3} (x_1\omega^{+}_1 + x_2\omega^{+}_2 + x_3\omega^{+}_3). 
 \label{CKGH0}
\end{align}
\end{subequations}
The CKY tensor is $Z=x_1\omega^{+}_1 + x_2\omega^{+}_2 + x_3\omega^{+}_3$.

Conversely, let $(M,g_{ab})$ be a hyper-K\"ahler manifold \eqref{HK0}, with a Killing vector $t^{a}$ that preserves the hyper-K\"ahler structure $(\omega^{+}_1, \omega^{+}_2, \omega^{+}_3)$ (i.e. $t^{a}$ is tri-holomorphic). Let $x_{i}$ be the corresponding Hamiltonian scalar fields: $t\lrcorner\,\omega^{+}_{i}=\d{x}_{i}$, and let $r:=(x_1^2 + x_2^2 + x_3^2)^{1/2}$. Then $\hat\kappa^{+}$ defined by \eqref{CKGH0} is a conformal K\"ahler structure for $(M,g_{ab})$. Introducing spherical coordinates by $x_2+\i x_3 = r\sin\theta \, e^{\i\varphi}$, $x_{1}=r\cos\theta$, a short calculation gives
\begin{align}
 \hat\kappa^{+} = \frac{1}{r^2} \left[(\d\tau+A)\wedge\d{r} - Vr^2\sin\theta\, \d\theta\wedge\d\varphi \right]. \label{CKGH}
\end{align}

Requiring \eqref{GH} to be complete, the geometry must be a gravitational multi-instanton \cite{GibbonsHawking} (also called multi-centred instanton): the function $V$ is 
\begin{align}
 V = \sum_{I=0}^{N}V_{I}, \quad 
 V_{0}=\text{const.}, \quad V_{i} \equiv \frac{1}{r_{i}} = [(x_1-a_i)^2+(x_2-b_i)^2+(x_3-c_i)^2]^{-1/2}
 \label{MI}
\end{align}
where $a_{i},b_{i},c_{i}$ are arbitrary constants, $i=1,...,N$, and $N$ is any positive integer\footnote{Instead of the $V_{i}$ in \eqref{MI}, one can also use $V_{i}\equiv 2m_i/r_i$ for $N$ constants $m_1,...,m_N$, but the singularities can be removed only if all the $m_i$'s are equal \cite{GibbonsHawking2}. The inclusion of the $m_i$'s does not affect our discussion.}. 
Defining spherical coordinates $(r_{i},\theta_{i},\varphi_{i})$ centred at $(a_i,b_i,c_i)$, the 1-form $A$ is $A=\sum_{I=0}^{N}A_{I}$, with $A_{0}=0$ and $A_{i}=\cos\theta_{i}\d\varphi_{i}$. If $V_0\neq0$ the geometry is ALF, and if $V_0=0$ (but $N\neq0$) it is ALE. 
The cases $(V_0, N)=(1,1),(0,2),(1,2)$ are, respectively, ASD Taub-NUT, Eguchi-Hanson, and double Taub-NUT.
The cases $V\equiv V_0$ and $V\equiv V_{i}$ (for $i$ fixed) both give flat space, so the instanton metrics \eqref{MI} are  formed by a suitable combination of flat pieces. Page proved \cite{Page79} that the Green's functions for multi-instantons also share this property. In this spirit, and connecting to the twistor discussion of section \ref{Sec:KSTT}, we have:

\begin{proposition}\label{prop:MI}
The conformal K\"ahler structure of the multi-instantons \eqref{MI} can be expressed as a combination of conformal K\"ahler structures on locally flat 4-spaces, each of which is an elementary state \eqref{ES} in the ALE case, and the same plus a dyon in the ALF case. However, the complex structures of different flat pieces are not compatible.
\end{proposition}

We interpret this result as saying that near each instanton, the geometry looks like an elementary state (apart from the dyon, which is some sort of ``background'' field), but the fact that the complex structures are not compatible means that the whole configuration is not simply a linear superposition of elementary states\footnote{Physically, if we equip a flat piece with the conformal K\"ahler form of a different flat piece, it will not satisfy the vacuum Maxwell equations.}. Rather, when there is more than one instanton they will ``interact'', giving origin to the curved geometry \eqref{GH}. The interaction could be described by trying to find a single flat space where \eqref{CKGH} is a complex structure. While what the corresponding twistor quadric may look like is not clear to us, the intuition behind this idea comes from the two-centred solution: see the discussion following equation \eqref{2centres} below (section \ref{Sec:OO}).

\begin{proof}[Proof of prop. \ref{prop:MI}]
We write \eqref{CKGH} as a sum $\hat\kappa^{+}=\sum_{I=0}^{N}\hat\kappa^{+}_{I}$, 
where 
\begin{subequations}\label{CKGHparts}
\begin{align}
\hat\kappa^{+}_{0} ={}& \frac{1}{r^2}\left[ \d\tau'\wedge\d{r} - V_0 r^2 \sin\theta \d\theta\wedge\d\varphi \right], \qquad \tau'\equiv\frac{\tau}{N+1}, \label{k0} \\
\hat\kappa^{+}_{i} ={}& \frac{1}{r^2}\left[ (\d\tau'+\cos\theta_{i}\d\varphi_{i})\wedge\d{r} - \frac{r^2}{r_{i}}\sin\theta \d\theta\wedge\d\varphi \right]. \label{ki}
\end{align}
\end{subequations} 
Note that if $V_0=0$ (ALE), we can take $\hat\kappa^{+}_{0}\equiv 0$ since $V_0=0=N$ is not allowed; and if $V_0\neq0$ (ALF) we can set $V_0\equiv1$. Define the tensor fields $g_{I}$, $I=0,...,N$, by
\begin{subequations}
\begin{align}
g_{0} ={}& \d\tau'^2+\d{r}^2+r^2\d\theta^2+r^2\sin^2\theta\d\varphi^2, \label{g0} \\
g_{i} ={}& r_i(\d\tau'+\cos\theta_i\d\varphi_i)^2 
 + \frac{1}{r_i} (\d{r}^2+r^2\d\theta^2+r^2\sin^2\theta\d\varphi^2). \label{gi}
\end{align}
\end{subequations}
Then $g_0$ and $g_{i}$ are locally flat metrics. For fixed $I$, $g_{I}$ is (locally) a Gibbons-Hawking metric \eqref{GH} with $V$ replaced by $V_{I}$, so for fixed $I$, $\hat\kappa^{+}_{I}$ is a conformal K\"ahler structure for $g_{I}$. In other words, we can form pairs $(g_{I},\hat\kappa^{+}_{I})$, each of which is a locally flat metric $g_{I}$ equipped with a conformal K\"ahler structure $\hat\kappa^{+}_{I}$. From section \ref{Sec:KSTT}, each $(g_{I},\hat\kappa^{+}_{I})$ is thus generated by a quadric surface in flat twistor space. 
Using prop. \ref{prop:CF}, to find the quadric corresponding to each $(g_{I},\hat\kappa^{+}_{I})$, we can focus on the conformal factor $\Omega_{+}$. For all $I$ we have $\Omega_{+}=r^{-1}$, but the coordinate $r$ has different meanings in the flat metrics $g_0,g_i$. We need to first find Cartesian coordinates for $g_{I}$, then express $r$ in terms of these coordinates, and then find the quadric that generates $r^{-1}$.

For $(g_0,\hat\kappa^+_0)$, we see from \eqref{g0} that $(r,\theta,\varphi)$ are spherical coordinates in $\mathbb{R}^3$ and $\tau'$ is Euclidean time, so Cartesian coordinates are $(\tau',x^1_0,x^2_0,x^3_0)$ with $x^1_0+\i x^2_0=r\sin\theta e^{\i\varphi}$, $x^3_0=r\cos\theta$, and $r = \left[(x^1_0)^2+(x^2_0)^2+(x^3_0)^2\right]^{1/2}$. Using \eqref{FormulaCF}, a twistor quadric generating $(g_0,\hat\kappa^+_0)$ is $Z^0 Z^3 - Z^1 Z^2$, which corresponds to a dyon (not an elementary state). 

For $(g_i,\hat\kappa^+_i)$, define $(x^0_i, x^1_i, x^2_i, x^3_i)$ by
\begin{align}\label{InerCoord}
 x^0_i+\i x^3_i = 2\sqrt{r_i}\sin(\theta_i/2)e^{\i(\tau' - \varphi_i)/2}, \qquad
 x^1_i+\i x^2_i = 2\sqrt{r_i}\cos(\theta_i/2)e^{\i(\tau' + \varphi_i)/2}.
\end{align}
Then a short calculation shows that $(\d{x}^0_i)^2+(\d{x}^1_i)^2+(\d{x}^2_i)^2+(\d{x}^3_i)^2$ is exactly the right hand side of \eqref{gi}; thus, $(x^0_i, x^1_i, x^2_i, x^3_i)$ are Cartesian coordinates for $g_{i}$. The $N$ arbitrary points $(a_1,b_1,c_1), ..., (a_N,b_N,c_N)$ represent the location of the instantons in Euclidean 3-space. For fixed $i$, we can choose the coordinate system $(x_1,x_2,x_3)$ in \eqref{GH}, \eqref{MI} to be centred at $(a_i,b_i,c_i)$, which amounts to taking $(a_i,b_i,c_i)\equiv(0,0,0)$ for that fixed $i$. This means that $r=r_{i}$, and from \eqref{InerCoord} we find
\begin{align}
 r = r_{i} = \frac{1}{4}\left[ (x^0_i)^2 + (x^1_i)^2 + (x^2_i)^2 + (x^3_i)^2 \right]. \label{ri}
\end{align}
Using \eqref{FormulaCF}, the quadric generating $(g_i,\hat\kappa^{+}_i)$ is the elementary state $Z^0Z^1$.

Finally, it remains to show that the complex structure of a flat piece is not a complex structure for the others. To do this, we recall from section \ref{Sec:KSTT} that the fundamental 2-form $\kappa$ necessarily has definite duality under the Hodge star of the metric. Thus, it suffices to show that, for $I\neq J$, the 2-form $\hat\kappa^{+}_{J}$ is not an eigenform of the Hodge star of $g_{I}$. This is a matter of calculation; one can see that
\begin{align}
*_{g_0}\hat\kappa^{+}_{1} \neq \pm \, \hat\kappa^{+}_{1}, 
\qquad 
*_{g_1}\hat\kappa^{+}_{2} \neq \pm \, \hat\kappa^{+}_{2},
\end{align}
where one can assume $(a_1,b_1,c_1)=(0,0,0)$ and $(a_2,b_2,c_2)=(0,0,c)$ with $c\neq 0$.
\end{proof}

\section{Opposite orientation}
\label{Sec:OO}

We assume now that $(M,g_{ab})$ is hyper-K\"ahler and has a conformal K\"ahler structure with opposite orientation. This means that there is a parallel spinor $o_{A'}$ \eqref{parallelspinor} and an unprimed Killing spinor $\sigma_{AB}$ \eqref{unprimedKS} with $\xi_{a}\neq0$. The integrability conditions for \eqref{unprimedKS} imply that the ASD Weyl spinor must be type D: denoting by $o_{A}, o^{\dagger}_{A}$ the principal spinors, with $o_Ao^{\dagger A}=1$, we have
\begin{align}
\Psi_{ABCD} = 6\Omega_{-}^{3}o_{(A}o_{B}o^{\dagger}_{C}o^{\dagger}_{D)}, \qquad
 \sigma_{AB}=2\i \Omega_{-}^{-1} o_{(A}o^{\dagger}_{B)}.
\label{typeD}
\end{align}
The SD and ASD parts of $\nabla_{a}\xi_{b}$ are respectively $\varphi_{AB}=\frac{1}{2}\nabla_{AA'}\xi^{A'}_{B}$ and $\chi_{A'B'}=\frac{1}{2}\nabla_{AA'}\xi^{A}_{B'}$. A short calculation gives $\varphi_{AB}= -\frac{1}{2}\Omega_{-}^3\sigma_{AB}$, so
\begin{align}
\nabla_{a}\xi_{b} = -\tfrac{1}{2}\Omega_{-}^3\sigma_{AB}\epsilon_{A'B'} + \chi_{A'B'}\epsilon_{AB}. \label{dKV}
\end{align}
Contracting eq. \eqref{unprimedKS} with $\sigma^{BC}$ and eq. \eqref{dKV} with $\xi^{b}$, we find the following identities:
\begin{subequations}\label{Ernst}
\begin{align}
 \sigma_{AB}\xi^{B}_{A'} ={}& \tfrac{1}{2}\nabla_{AA'}\mathcal{E}_{-}^{-2}, 
  \qquad \mathcal{E}_{-}:=-\Omega_{-} \label{sigmachi} \\
 \chi_{A'B'}\xi^{B'}_{A} ={}& -\tfrac{1}{2}\nabla_{AA'}\mathcal{E}_{+}, 
 \qquad \mathcal{E}_{+}:=\Omega_{-}+\xi_b\xi^b. \label{chixi}
\end{align}
\end{subequations}
The notation $\mathcal{E}_{\pm}$ is chosen so as to illustrate that these scalar fields are the familiar Ernst potentials, whose existence follows from the facts that $\xi^a$ is Killing and the space is Ricci-flat. The Killing equation for $\xi_{a}$ leads to $\nabla_{a}\nabla_{b}\xi_{c} = R_{bca}{}^{d}\xi_{d}$. Using \eqref{HK0}, this gives
\begin{align}
\nabla_{AA'}\chi_{B'C'} = 0. \label{dchi}
\end{align}
We can then separate into two cases, according to whether $\chi_{A'B'}$ vanishes or not. The case $\chi_{A'B'}\equiv0$ leads to the following:

\begin{proposition}\label{prop:TaubNUT}
The only hyper-K\"ahler metric with a non-constant Killing-Yano tensor is the half-flat Taub-NUT instanton.
\end{proposition}

\begin{proof}
Suppose that $(M,g_{ab})$ is hyper-K\"ahler and has a non-constant KY tensor $Y_{ab}$. From prop. \ref{prop:KY}, the KY equation is equivalent to \eqref{KYEq}. Thus there is a primed Killing spinor $K_{A'B'}$, so we know from section \ref{Sec:HKO} that its Hamiltonian vector field $\xi_{a}$ is ASD. But there is also an unprimed Killing spinor $\sigma_{AB}$ with the same $\xi_{a}$ (eq. \eqref{KYEq}), which then satisfies \eqref{dKV} and, since it is ASD, we must have $\chi_{A'B'}\equiv0$. So it remains to see that the geometry must be the half-flat (ASD in our conventions) Taub-NUT solution. 

To find the metric, we can use the formulation given by Tod in \cite{Tod2020} (we follow the details of \cite{A23b}): defining $z:=\Omega_{-}^{-1}$, $W:=(\xi_a\xi^a)^{-1}$, $\xi^{a}\partial_{a}=\partial_{\psi}$, there exist coordinates $x,y$ and a function $u(x,y,z)$ such that the metric $g$ and the K\"ahler form $\hat\kappa^{-}$ are
\begin{align}
 g ={}& W^{-1}(\d\psi+A)^2+W[\d{z}^2+e^{u}(\d{x}^2+\d{y}^2)], \label{gToda} \\
\hat\kappa^{-}={}& (\d\psi+A)\wedge\frac{\d{z}}{z^2} + \frac{We^u}{z^2}\d{x}\wedge\d{y}, 
\label{ckToda}
\end{align}
where the 1-form $A$ is defined by $\xi_{a}\d{x}^a=W^{-1}(\d\psi+A)$. 
The Ricci-flat condition implies that $u$ satisfies the $SU(\infty)$ Toda equation
\begin{align}
 u_{xx}+u_{yy}+(e^{u})_{zz} = 0. \label{TodaEq}
\end{align}

From \eqref{chixi} with $\chi_{A'B'}=0$, we see that the Ernst potential $\mathcal{E}_{+}$ must be constant, and this constant cannot vanish since $\Omega_{-}>0$, $\xi_a\xi^a >0$. For  convenience we denote 
\begin{align*}
 \mathcal{E}_{+} \equiv (2n)^{-2/3}.
\end{align*}
Recalling the expression \eqref{chixi} for $\mathcal{E}_{+}$, solving for $W=(\xi_a\xi^a)^{-1}$, and introducing a new coordinate $\rho$ by $z=\frac{1}{(2n)^{1/3}}(\rho+n)$, we find $W=(2n)^{2/3}\frac{(\rho+n)}{(\rho - n)}$.
Now, using that $\hat\kappa^{-}_{ab}=\Omega_{-}^{3}\sigma_{AB}\epsilon_{A'B'}$, equation \eqref{dKV} with $\chi_{A'B'}=0$ can be written as $\d\xi=-\hat\kappa^{-}$. Comparing this to \eqref{ckToda}, we find $\d{A}=-\frac{W^2e^u}{z^2}\d{x}\wedge\d{y}$. 
The equation $\d^2A=0$ leads to $\frac{W^2e^u}{z^2} = (2n)^{2/3}h(x,y)$ for some function $h(x,y)$, where we include the factor $(2n)^{2/3}$ for convenience. This gives $e^{u} = (2n)^{-2/3} (\rho-n)^2 h(x,y)$. Replacing in the Toda equation \eqref{TodaEq}, we arrive at 
$R_{\Sigma} := - h^{-1}(\partial_x^2+\partial_y^2)\log h = 2$. The notation $R_{\Sigma}$ is meant to illustrate that $R_{\Sigma}$ can be thought of as the scalar curvature of the 2-metric $g_{\Sigma}:=h(x,y)(\d{x}^2+\d{y}^2)$ in \eqref{gToda}. By the uniformization theorem, $R_{\Sigma}=2$ implies that $g_{\Sigma}$ must be the metric of a round 2-sphere: $h(x,y)(\d{x}^2+\d{y}^2)=\d\theta^2+\sin^2\theta\d\varphi^2$. So we can set $h=\sin^2\theta$, $\d{x}=\frac{\d\theta}{\sin\theta}$, $\d{y}=\d\varphi$. Using the equation for $\d{A}$ we found before, we deduce $A=(2n)^{4/3}\cos\theta\d\varphi$. Finally, defining $\tau:=\psi/(2n)^{4/3}$ and replacing everything in \eqref{gToda}, the metric is
\begin{align}
 g = 4n^2\left( \frac{\rho-n}{\rho+n} \right)(\d\tau+\cos\theta\d\varphi)^2 
 + \left( \frac{\rho+n}{\rho-n} \right)\d\rho^2 + (\rho^2-n^2)(\d\theta^2+\sin^2\theta\d\varphi^2),
 \label{TaubNUT}
\end{align}
which we recognise as the ASD Taub-NUT instanton, with NUT charge $n$. 
\end{proof}

We now assume that $\chi_{A'B'}\neq0$. This means that the Killing vector $\xi^{a}$ is not ASD. The interesting structures arising in this case are summarised in the following:

\begin{proposition}\label{prop:KT}
Suppose that the Killing vector $\xi^{a}$ is not ASD. Define a symmetric tensor $H_{ab}$ and a vector field $t^{a}$ by
\begin{align}
H_{ab}:=\sigma_{AB}\chi_{A'B'} - \tfrac{1}{2}\mathcal{E}_{+} g_{ab}, 
\qquad t_{a}:=H_{ab}\xi^{b}. \label{KTdef}
\end{align}
Then:
\begin{align}
 \nabla_{(a}H_{bc)} = 0, \qquad \nabla_{(a}t_{b)} = 0, \qquad \nabla_{AA'}t^{A}_{B'}=0, 
 \qquad [t,\xi] = 0. \label{KT}
\end{align}
\end{proposition}

\begin{proof}
To show that $H_{ab} $ is a Killing tensor, use equations \eqref{unprimedKS}, \eqref{dchi} and \eqref{chixi} to compute $\nabla_{a}H_{bc}$, then $\nabla_{(a}H_{bc)}=0$ follows easily. To show that $t_{a}$ is a Killing vector, we first show that $\xi^{c}\nabla_{c}H_{ab}=0$ and $\pounds_{\xi}H_{ab}=0$. The latter equation also implies that $t^a$ and $\xi^a$ commute. Finally, the part that requires more computation is the ASD property of $t_{a}$. To show this, we first derive the identity
\begin{align*}
\nabla_{AA'}t^{A}_{B'} = \left[ \tfrac{1}{4}\Box\Omega_{-}^{-2} - \Omega_{-} - \tfrac{3}{2}\xi_c\xi^c \right] \chi_{A'B'},
\end{align*}
which follows from the definitions \eqref{KTdef}. Applying $\nabla^{AA'}$ to \eqref{sigmachi} and using $\nabla^{AA'}\sigma_{AB}=-3\xi^{A'}_{B}$ and $\nabla^{AA'}\xi_{A'}^{B}=\Omega_{-}^3\sigma^{AB}$, we get $\Box\Omega_{-}^{-2} = 4\Omega_{-} +6\xi_{a}\xi^{a}$, so $\nabla_{AA'}t^{A}_{B'}=0$ follows.
\end{proof}

We then see that there is a Killing tensor and a second Killing vector $t_{a}$ commuting with $\xi_{a}$. In addition, $t_{a}$ is ASD (so it is not proportional to $\xi_{a}$), which implies that the metric must be Gibbons-Hawking, with $t^{a}$ the tri-holomorphic Killing field. The space is then again conformally K\"ahler w.r.t. both orientations. The second Killing vector arises via a Hughston-Sommers-type construction \cite{Hughston}, even though there are no (non-constant) KY tensors. The metric is \eqref{GH} with $V=(t_{a}t^{a})^{-1}$ and $t^{a}$ defined in \eqref{KTdef}.

In \cite{PH}, Haslehurst and Penrose argue that the most general type D, half-flat complex space is the ASD (in our conventions) limit of the Pleba\'nski-Demia\'nski solution \cite{PD}. Since ASD vacuum type D implies that there is an unprimed Killing spinor, this is the same as our assumptions at the beginning of this section, so our $(M,g_{ab})$ must be the solution in \cite{PH}. Furthermore, \cite{Woodhouse} argues that the solution is a two-centred, ALE Gibbons-Hawking multi-instanton (see also \cite{Casteill}), with $V$ given by 
\begin{align}
 V = \frac{m_1}{r_1} + \frac{m_2}{r_2} 
 \label{2centres}
\end{align}
for some constants $m_1,m_2$ (and $r_1,r_2$ are defined in \eqref{MI}). The SD conformal K\"ahler form $\hat\kappa^{+}$ is \eqref{CKGH}, and from proposition \ref{prop:MI} it is described, near each of the two instantons, by an elementary state. But we know that when both of the instantons are taken into account, the twistor quadric must be different. (Intuitively, each instanton will ``feel'' the field of the other.) Now, when there are only two instantons, we can find the solution to this issue, since the situation can be described in terms of the relative distance of the two centres. We can take $(a_1,b_1,c_1)=(0,0,0)$ and $(a_2,b_2,c_2)=(0,0,c_2)$, then $r_1=r$ and 
\begin{align}
 r = \sqrt{ r_2^2 + 2c_2 r_2\cos\theta_2 + c_2^2}. \label{r2}
\end{align}
The Cartesian coordinates of instanton 2, $(x_2^0,x_2^1,x_2^2,x_2^3)$, are given by \eqref{InerCoord} with $i=2$, from where we deduce
\begin{align*}
 r_2=\frac{1}{4}\left[ (x^0_2)^2 + (x^1_2)^2 + (x^2_2)^2 + (x^3_2)^2 \right], \quad
 r_2\cos\theta_2=\frac{1}{4}\left[ (x^1_2)^2 + (x^2_2)^2 - (x^0_2)^2 - (x^3_2)^2 \right].
\end{align*}
Replacing in \eqref{r2}, and defining $c_2\equiv c^2/4$ for convenience, we get
\begin{align}
 r = \frac{1}{4}\sqrt{\left[\delta_{ab} x_2^a x_2^b - c^2\right]^2 + 4c^2\left[(x^1_2)^2 + (x^2_2)^2\right]}.
 \label{rPD}
\end{align}
Recalling that $r=\Omega_{+}^{-1}$ and using \eqref{FormulaCF}, we find that the twistor quadric generating \eqref{rPD} is $Z^0Z^1 + \frac{c^2}{2} Z^2Z^3$. This is not an elementary state. But the conformal K\"ahler form defined by this quadric is our original $\hat\kappa^{+}$ (i.e. \eqref{CKGH} with \eqref{2centres}), since the quadric generates the conformal factor (recall prop. \ref{prop:CF}). In other words: this argument gives us the flat-space twistor quadric associated to the two instantons. Physically, the different structure of the singularity can be seen by Wick-rotating\footnote{In our conventions, a Wick rotation corresponds to $x_2^0=t$, $x_2^1=\i x$, $x_2^2=\i y$, $x_2^3= \i z$.} to Minkowski space-time, where $r=0$ in \eqref{rPD} describes two objects undergoing hyperbolic motion (instead of the hopfion-like field of an elementary state), see \cite{A23a}. This gives us the ``interaction'' of the two instantons, and is precisely the singular structure of the Pleba\'nski-Demia\'nski space-time \cite{PD}.

A possible objection to the above picture, is that the same solution \eqref{rPD} is obtained if one considers only {\em one} instanton, say $(a_1,b_1,c_1)$, but chooses $(a_1,b_1,c_1)=(0,0,c_1)$ instead of all zero. From prop. \ref{prop:MI} one expects a single instanton to correspond to an elementary state, and yet we just argued that \eqref{rPD} describes a Pleba\'nski-Demia\'nski field. But the point is that if $(a_1,b_1,c_1)=(0,0,c_1)$, one is describing the field relative to a reference system that is not located at the instanton.
Intuitively, the situation may be compared to electromagnetism, where an observer situated at a point-like charge sees only an electric field, whereas an observer in relative motion will also see a magnetic field. 
In our case, a reference system located at the instanton will see an elementary state, but a system away from it will see the field \eqref{rPD}. When there is a single instanton one can choose where to put the coordinate system, but when there are two, from either location one will see the field of the other.

\section{Final remarks}

We saw that a hyper-K\"ahler instanton with an equally or oppositely oriented conformal K\"ahler structure must be given by the Gibbons-Hawking ansatz. This leaves out some geometries, notably the Atiyah-Hitchin metric. We decomposed the conformal K\"ahler structure into a number of locally flat pieces, where each piece can be described by a twistor elementary state (plus a background dyon in the ALF case), but the different complex structures are not compatible, reflecting the fact that the geometry is not a linear superposition and the instantons `interact'. 

A possible description in terms of a single flat space with a single complex structure is motivated by the situation in (non-self-dual) black hole solutions. As a simple example, for the Schwarzschild solution in standard Schwarzschild coordinates, the conformal K\"ahler 2-form is independent of mass and thus ``lives'' on flat space. This was used in \cite{A22} to argue that the space-time can be locally described as a deformed twistor quadric, and the same can be done for Kerr and for more general black hole space-times. 

For the multi-instantons studied in the current note, we saw that when there are only two centres, the system can be described in terms of their relative distance in $\mathbb{R}^3$; we used this to find a single flat space compatible with the complex structure, and we saw that the description is consistent with the fact that the solution must be Pleba\'nski-Demia\'nski. When there are more than two instantons this trick cannot be applied as there are more parameters, but a similar description in terms of a single flat space would allow to Wick rotate the linearised instantons to describe their interaction in Minkowski space-time; in particular, in the three-centred Gibbons-Hawking case, one could get a spin 2 field in Minkowski representing the linearised Chen-Teo solution.

\end{document}